\documentclass{siamltex}
\usepackage{hyperref}
\usepackage{color}
\usepackage[numbers]{natbib}
\usepackage{amssymb}
\usepackage{amsmath}
\usepackage{booktabs}
\usepackage{amsfonts}
\usepackage{graphicx}

\definecolor{DarkBlue}{rgb}{0.00,0.00,0.55}
\definecolor{Black}{rgb}{0.00,0.00,0.00}
\hypersetup{
    linkcolor = DarkBlue,
    anchorcolor = DarkBlue,
    citecolor = DarkBlue,
    filecolor = DarkBlue,
    urlcolor = DarkBlue,
    colorlinks  = true,
}

\newcommand{\dx}{\ \mathrm{d}x}

\title{A preconditioner for the Ohta--Kawasaki equation}
\author{
  Patrick~E.~Farrell\thanks{Mathematical Institute, University of Oxford, Oxford, UK.
    Center for Biomedical Computing, Simula Research Laboratory, Oslo, Norway
    (\texttt{patrick.farrell@maths.ox.ac.uk}).}
  \and
  John~W.~Pearson\thanks{School of Mathematics, Statistics and Actuarial Science, University of Kent, Canterbury CT2 7NF, UK (\texttt{j.w.pearson@kent.ac.uk}).
This research is funded by EPSRC grants EP/K030930/1, EP/M018857/1, and a Center
of Excellence grant from the Research Council of Norway to the Center for
Biomedical Computing at Simula Research Laboratory. This work used the ARCHER UK
National Supercomputing Service (http://www.archer.ac.uk), via a RAP award
to E.~S\"uli and Q.~Parsons. We thank NOTUR for
the allocation of computing resources on Hexagon. The authors would like to
acknowledge the assistance of C.~N.~Richardson, J.~Ring, M.~F.~Adams and
G.~N.~Wells in conducting the numerical experiments.}
  }

\begin{document}
\maketitle

\begin{abstract}
We propose a new preconditioner for the Ohta--Kawasaki equation, a nonlocal
Cahn--Hilliard equation that describes the evolution of diblock copolymer melts.
We devise a computable approximation to the inverse of the Schur complement of
the coupled second-order formulation via a matching strategy. The preconditioner
achieves mesh independence: as the mesh is refined, the number of Krylov
iterations required for its solution remains approximately constant. In
addition, the preconditioner is robust with respect to the interfacial thickness
parameter if a timestep criterion is satisfied. This enables the highly resolved
finite element simulation of three-dimensional diblock copolymer melts with over
one billion degrees of freedom.
\end{abstract}



\section{Introduction}
The Ohta--Kawasaki equations \cite{ohta1986} model the evolution of diblock copolymer melts.
A diblock copolymer is a polymer consisting of two subchains of different
monomers that repel each other but are joined by a covalent bond. A large
collection of these molecules is termed a melt.
These melts are of scientific and engineering interest because they undergo
phase separation of their different constituent monomers, allowing for the
design of nanostructures with particular desirable properties. The numerical
simulation of these equations is an essential tool in exploring the
associated phase diagram \cite{choksi2009}.

The Ohta--Kawasaki functional describes the free energy of a
diblock copolymer melt:
\begin{equation}
E[u] = \frac{1}{2} \int_\Omega \left( \varepsilon^2 \left| \nabla u \right|^2  \vphantom{\left|(-\Delta_N)^{-1/2}(u - m)\right|^2}\right.
     + \frac{1}{2} \left(1 - u^2\right)^2
     + \left. \sigma \left|(-\Delta_N)^{-1/2}(u - m)\right|^2\right) \dx,
\end{equation}
where $u = \pm 1$ denotes the two pure phases, $\Omega$ is the domain under
study ($(0, 1)^2$ or $(0, 1)^3$), $\varepsilon \ll 1$ is the interfacial thickness
between regions of the pure phases, $\sigma$ is the nonlocal energy coefficient,
$m$ is the (conserved) average value of $u$ in $\Omega$, and $\Delta_N$ denotes the Laplacian
with homogeneous Neumann boundary conditions. Assuming that the dynamics are
governed by a $H^{-1}(\Omega)$ gradient flow,
\begin{equation}
(u_t, \phi)_{H^{-1}(\Omega)} + E'[u; \phi] = 0,
\end{equation}
the resulting Ohta--Kawasaki dynamic equation on $\Omega \times (0, T]$ in second-order
form is given by \cite{parsons2012}
\begin{subequations}
\begin{align}
u_t - \Delta w + \sigma (u - m) &= 0, \label{eqn:oka} \\
w + \varepsilon^2 \Delta u - u(u^2 - 1) &= 0, \label{eqn:okb}
\end{align}
\end{subequations}
with homogeneous Neumann boundary conditions
\begin{equation} \label{eqn:okbcs}
\nabla u \cdot n = 0 \text{ and } \nabla w \cdot n = 0,
\end{equation}
and with initial condition $u(x, 0) = u_0(x)$.

\section{Approximating the Schur complement}

After applying a finite element discretization in space and the $\theta$-method
in time, a discrete nonlinear problem must be solved at each timestep. Each
Newton iteration involves solving a linear system of the form
\begin{equation} \label{eqn:newton}
J 
\begin{bmatrix}
\delta u \\
\delta w \\
\end{bmatrix}
=
\begin{bmatrix}
(1 + \Delta t \theta \sigma) M & \Delta t \theta K \\
-\varepsilon^2 K - M_E            & M
\end{bmatrix}
\begin{bmatrix}
\delta u \\
\delta w \\
\end{bmatrix}
=
\begin{bmatrix}
f_1 \\
f_2 \\
\end{bmatrix}.
\end{equation}
where $J$ is the Jacobian, $M$ is the standard mass matrix with entries of the form $\int_{\Omega}\phi_i\phi_j \dx$, $K$ is
the standard discretization of the Neumann Laplacian with entries $\int_{\Omega}\nabla\phi_i\cdot\nabla\phi_j \dx$, 
$M_E$ is a mass matrix involving a spatially varying coefficient with entries
$\int_{\Omega}(3u^2 - 1)\phi_i\phi_j \dx$,
$\delta u$ is the update for $u$, $\delta w$ is the update for $w$, $f_1$ and $f_2$
gather the source term and contributions from previous time levels, and
$\Delta t$ is the timestep. As the
discretization is refined and the dimension of \eqref{eqn:newton} increases, it
becomes impractical to employ direct solvers and preconditioned Krylov methods
must be used instead. As the matrix in \eqref{eqn:newton} is nonsymmetric, a suitable
iterative solver such as GMRES
\cite{saad1986} is required to compute its solution.

We note however that there are structures within the system that can be exploited within a
solver. For example, $M$ is symmetric positive definite, $K$ is symmetric positive
semidefinite (with one zero eigenvalue corresponding to the nullspace of constants),
and $M_E$ is symmetric.

Preconditioners for block-structured matrices typically involve approximating
the Schur complement of the system. Let the Jacobian $J$ be partitioned as
\begin{equation}
J = 
\begin{bmatrix}
A & B \\
C & D
\end{bmatrix}.
\end{equation}
Consider the preconditioner
\begin{equation} \label{eqn:prec}
P^{-1} = 
\begin{bmatrix}
A & 0 \\
C & S
\end{bmatrix}^{-1}
=
\begin{bmatrix}
A^{-1} & 0 \\
0      & S^{-1}
\end{bmatrix}
\begin{bmatrix}
I & 0 \\
-CA^{-1} & I
\end{bmatrix},
\end{equation}
where $S = D - CA^{-1}B$ is the Schur complement with respect to $A$.
If exact inner solves are used for $A^{-1}$ and $S^{-1}$, then the
preconditioned operator $P^{-1} J$ has minimal polynomial degree two and GMRES converges
exactly in two iterations \cite{ipsen2001,murphy2000}. Here,
$A$ is a scaled mass matrix and is thus straightforward to solve with
standard techniques such as Chebyshev semi-iteration. Thus, the efficient preconditioning of $S$ is
the key to the fast solution of the Newton step \eqref{eqn:newton}, and thus of the
dynamic Ohta--Kawasaki equations \eqref{eqn:oka}--\eqref{eqn:okb}.

The Schur complement of \eqref{eqn:newton} is
\begin{equation} \label{eqn:schur}
S = M + \varepsilon^2 c K M^{-1} K
+ c M_E M^{-1} K,
\end{equation}
with constant $c = {(\Delta t \theta)}/{(1 + \Delta t \theta \sigma)} > 0$.
In general it is very difficult to precondition the sum of different matrices.
The approach adopted here is the \emph{matching strategy} of Pearson and Wathen
\cite{pearson2012,bosch2014}: the sum is approximated by the product of
matrices, carefully chosen to match as many terms of the sum as possible.
We propose the approximation 
\begin{equation} \label{eqn:outer_approx}
S \approx \tilde{S} = \hat{S} M^{-1} \hat{S},
\end{equation}
with
\begin{equation}
\hat{S} = M + \varepsilon \sqrt{c} K.
\end{equation}
This approximation $\tilde{S}$ is the product of three invertible matrices, and
so its inverse action can be efficiently computed by 
\begin{equation} \label{eqn:schurinv}
\tilde{S}^{-1} = \hat{S}^{-1} M \hat{S}^{-1}.
\end{equation}
The action of $\hat{S}^{-1}$ can be
efficiently approximated with algebraic multigrid techniques \cite{henson2002} to yield
a computationally cheap preconditioner.

Expanding the approximation \eqref{eqn:outer_approx}, we find that it matches the first two terms of the
Schur complement exactly:
\begin{equation} \label{eqn:schurapprox}
\tilde{S} = M + \varepsilon^2 c K M^{-1} K
+ 2\varepsilon \sqrt{c} K.
\end{equation}
Recall that $M_E$ depends on the current estimate of the solution $u$.  With
this matching strategy, the term involving $M_E$ in the Schur complement
\eqref{eqn:schur} has been neglected, so that $\tilde{S}$ does not vary between
Newton iterations and no reassembly or algebraic multigrid reconstruction is
required. The Schur complement approximation $\tilde{S}$
is straightforward and feasible to apply, and its effectiveness within a
preconditioner will depend to a large extent on the effect of the neglected
third term from $S$. In the next section we present some analysis to explain
why we expect our approximation to work well as a preconditioner.

We note in passing that it is possible to rearrange \eqref{eqn:newton} to the
symmetric form
\begin{equation*}
\begin{bmatrix}
\frac{\Delta t \theta}{1 + \Delta t \theta \sigma} K & M \\
M & -\varepsilon^2 K - M_E
\end{bmatrix}
\begin{bmatrix}
\delta w \\
\delta u \\
\end{bmatrix}
=
\begin{bmatrix}
\frac{1}{1 + \Delta t \theta \sigma} f_1 \\
f_2 \\
\end{bmatrix}.
\end{equation*}
It is likely that a similar approach could be used to precondition this
rearranged system, providing one takes into account the singular $(1,1)$-block.
Furthermore, as discussed in the next section, it is more straightforward to
prove the rates of convergence of iterative methods for symmetric systems.
However, whereas the $(1,1)$-block $A$ of \eqref{eqn:newton} may be well
approximated using Chebyshev semi-iteration, the stiffness matrix $K$ arising in
the $(1,1)$-block of the rearranged system would require a more expensive method
such as a multigrid process. We therefore prefer to solve the nonsymmetric system
\eqref{eqn:newton} due to the ease with which we may compute the
approximate action of $A^{-1}$.

\section{Analysis}
We now wish to justify why our preconditioner is likely to be
effective for the problem being solved. It is well-known that for nonsymmetric
matrix systems it is extremely difficult to provide a concrete proof for the
rate of convergence of an iterative method, as opposed to symmetric systems for
which convergence is controlled only by the eigenvalues of the preconditioned
system $P^{-1}J$. For nonsymmetric operators the spectrum alone does not determine
the convergence of GMRES \cite{greenbaum1996}; furthermore, the usual techniques
for establishing eigenvalue bounds do not apply to the nonsymmetric case
\cite{bosch2014}. However, in practice, the tight clustering of eigenvalues for
the preconditioned system frequently leads to strong convergence properties,
even though this is not theoretically guaranteed. Given the challenges faced
when solving nonsymmetric matrix systems, we present an analysis that
establishes spectral equivalence of a slightly perturbed operator, and
corroborate this analysis with numerical experiments in section
\ref{sec:experiments} which demonstrate mesh independence for the full problem.

Observe that when applying the preconditioner
\begin{equation}
\ \begin{bmatrix}
\tilde{A} & 0 \\
C & \tilde{S}
\end{bmatrix}^{-1}
=
\begin{bmatrix}
\tilde{A}^{-1} & 0 \\
0      & \tilde{S}^{-1}
\end{bmatrix}
\begin{bmatrix}
I & 0 \\
-C\tilde{A}^{-1} & I
\end{bmatrix}
\end{equation}
for the Jacobian $J$, the crucial steps are applying $\tilde{A}^{-1}$ and
$\tilde{S}^{-1}$. For the matrix system \eqref{eqn:newton}, the inverse of the
sub-block $A=(1+\Delta{}t\theta\sigma)M$ may be approximated accurately and
cheaply using a Jacobi iteration or Chebyshev semi-iteration
\cite{golub1961a,golub1961b,wathen2009}. It is therefore
instructive to consider the preconditioned system
\begin{equation}
\ \begin{bmatrix}
A & 0 \\
C & \tilde{S}
\end{bmatrix}^{-1}
\begin{bmatrix}
A & B \\
C & D
\end{bmatrix}
=
\begin{bmatrix}
A^{-1} & 0 \\
0      & \tilde{S}^{-1}
\end{bmatrix}
\begin{bmatrix}
I & 0 \\
-CA^{-1} & I
\end{bmatrix}
\begin{bmatrix}
A & B \\
C & D
\end{bmatrix}
=
\begin{bmatrix}
I & A^{-1}B \\
0 & \tilde{S}^{-1}S
\end{bmatrix},
\end{equation}
where the inexactness of our preconditioner arises from the stated approximation
$\tilde{S}$ of the Schur complement $S$. The eigenvalues of this system, which
serve as a guide as to the effectiveness of the preconditioner, are either equal
to $1$, or correspond to the eigenvalues of $\tilde{S}^{-1}S$.

We therefore examine the spectrum of the
preconditioned Schur complement $\tilde{S}^{-1}S$, the matrix which
governs the effectiveness of our algorithm, in the ideal setting where
the matrix $M+\varepsilon \sqrt{c}K$ is inverted exactly. We first present a short result
concerning the reality of the eigenvalues.

\begin{lemma}
The eigenvalues of $\tilde{S}^{-1} S$ are real.
\end{lemma}
\begin{proof}
If $v \in \textrm{null}(K)$ and $v \neq 0$, then
\begin{equation} \label{eqn:geneig}
           Sv = \lambda \tilde{S} v
\end{equation}
implies
\begin{equation}
           Mv = \lambda M v,
\end{equation}
and hence $\lambda = 1 \in \mathbb{R}$.

On the other hand, if $v \notin \textrm{null}(K)$, then \eqref{eqn:geneig} implies
\begin{equation}
\ KM^{-1}Sv=\lambda{}KM^{-1}\tilde{S}v~~\Rightarrow~~F v = \lambda G v~~\Rightarrow~~v^* F v = \lambda v^* G v,
\end{equation}
where the matrices $F$ and $G$ are given by
\begin{align}
F &=  K + \varepsilon^2 cK M^{-1} K M^{-1} K + c K M^{-1} M_F M^{-1} K, \\
G &= K + \varepsilon^2 c KM^{-1}KM^{-1}K + 2 \varepsilon \sqrt{c}KM^{-1}K.
\end{align}
Using the symmetry of $F$ and $G$, it follows that $v^* F v$ and $v^* G v$ are
real and positive, and hence $\lambda \in \mathbb{R}$.
\end{proof}

To motivate why $\tilde{S}$ should serve as an effective approximation of $S$,
we now present a result on the eigenvalues of $\tilde{S}^{-1}S$ in the perturbed
setting that $K$ is symmetric positive definite. (The matrix $K$ is in practice
positive semidefinite, with a nullspace of dimension one.)
\begin{lemma}
Perturb $K$ to be symmetric positive definite. Then the eigenvalues of
$\tilde{S}^{-1}S$ satisfy:
\begin{equation}
\ \lambda(\tilde{S}^{-1}S)\in\left[\frac{1}{2}+\frac{1}{2\varepsilon}\hspace{0.1em}\sqrt{c}\hspace{0.1em}\lambda_{-},1+\frac{1}{2\varepsilon}\hspace{0.1em}\sqrt{c}\hspace{0.1em}\lambda_{+}\right),
\end{equation}
where $\lambda_{-}=\min\{\lambda_{\min}(M^{-1}M_E),0\}$, $\lambda_{+}=\max\{\lambda_{\max}(M^{-1}M_E),0\}$ respectively, and $\lambda_{\min}$, $\lambda_{\max}$ are the minimum and maximum
eigenvalues of a matrix.
\end{lemma}
\begin{proof}
First note that the eigenvalues of $\tilde{S}^{-1}S$ and $S\tilde{S}^{-1}$ are
the same by the similarity of the two matrices. We therefore examine the matrix
$S\tilde{S}^{-1}$, using the properties that $M$ is symmetric positive definite,
and that $M_E$ is symmetric. Using our assumption on $K$, $K^{-1}$ exists, and
it can be shown that
\begin{align}
\ S\tilde{S}^{-1}={}&\left(M + \varepsilon^2 c K M^{-1} K + c M_E M^{-1} K\right)\left(M + \varepsilon^2 c K M^{-1} K + 2\varepsilon \sqrt{c} K\right)^{-1} \\
\ ={}&I+\left(- 2\varepsilon \sqrt{c} K + c M_E M^{-1} K\right)\left(M + \varepsilon^2 c K M^{-1} K + 2\varepsilon \sqrt{c} K\right)^{-1} \\
\ ={}&I+\left(-2\varepsilon \sqrt{c}M+cM_E\right)M^{-1}K\left(M + \varepsilon^2 c K M^{-1} K + 2\varepsilon \sqrt{c} K\right)^{-1} \\
\ \label{eqn:equiv_schur} ={}&I+\left(-2\varepsilon \sqrt{c}M+cM_E\right)\left(MK^{-1}M+\varepsilon^2 cK+2\varepsilon \sqrt{c}M\right)^{-1}.
\end{align}
Therefore, each eigenvalue of $S\tilde{S}^{-1}$ is given by
$1+\lambda_{\mathcal{R}}$, where $\lambda_{\mathcal{R}}$ may be bounded using
a Rayleigh quotient argument, by symmetry of the matrices in \eqref{eqn:equiv_schur}.
The Rayleigh quotient is given by
\begin{align}
\ \mathcal{R}:={}&\frac{v^{T}\left(-2\varepsilon \sqrt{c}M+cM_E\right)v}{v^{T}\left(MK^{-1}M+\varepsilon^2 cK+2\varepsilon \sqrt{c}M\right)v} \\
\ ={}&\underbrace{-2\varepsilon \sqrt{c}\hspace{0.15em}\frac{v^{T}Mv}{v^{T}\left(MK^{-1}M+\varepsilon^2 cK+2\varepsilon \sqrt{c}M\right)v}}_{\mathcal{R}_{1}} \\
\ &+\underbrace{c\hspace{0.15em}\frac{v^{T}M_Ev}{v^{T}\left(MK^{-1}M+\varepsilon^2 cK+2\varepsilon \sqrt{c}M\right)v}}_{\mathcal{R}_{2}}.
\end{align}

First observe that $\mathcal{R}_{1}<0$. To find a lower bound on
$\mathcal{R}_{1}$, we note that
$\mathcal{R}_{1}=-2a^{T}b/(a^{T}a+b^{T}b+2a^{T}b)$, where $a=K^{-1/2}Mv$,
$b=\varepsilon \sqrt{c}K^{1/2}v$, and therefore $\mathcal{R}_{1}\geq-\frac{1}{2}$ by
simple algebraic manipulation. Therefore $\mathcal{R}_{1}\in[-\frac{1}{2},0)$.

Now note that $\mathcal{R}_{2}$ could be positive or negative, depending on the
sign of $v^{T}M_Ev$. If this quantity is positive for some $v$,
\begin{equation}
\ \mathcal{R}_{2}<\frac{c\hspace{0.1em}v^{T}M_Ev}{2\varepsilon \sqrt{c}\hspace{0.1em}v^{T}Mv}=\frac{\sqrt{c}}{2\varepsilon}\frac{v^{T}M_Ev}{v^{T}Mv}\leq\frac{\sqrt{c}}{2\varepsilon}\hspace{0.1em}\lambda_{\max}(M^{-1}M_E),
\end{equation}
and otherwise $\mathcal{R}_{2}\leq0$. Similarly, if $\mathcal{R}_{2}$ is negative for some $v$, then
\begin{equation}
\ \mathcal{R}_{2}>\frac{c\hspace{0.1em}v^{T}M_Ev}{2\varepsilon \sqrt{c}\hspace{0.1em}v^{T}Mv}=\frac{\sqrt{c}}{2\varepsilon}\frac{v^{T}M_Ev}{v^{T}Mv}\geq\frac{\sqrt{c}}{2\varepsilon}\hspace{0.1em}\lambda_{\min}(M^{-1}M_E),
\end{equation}
and $\mathcal{R}_{2}\geq0$ otherwise. Hence $\mathcal{R}_{2}\in[\frac{1}{2\varepsilon}\hspace{0.1em}\sqrt{c}\hspace{0.1em}\lambda_{-},\frac{1}{2\varepsilon}\hspace{0.1em}\sqrt{c}\hspace{0.1em}\lambda_{+}]$.

Combining the bounds for $\mathcal{R}_{1}$ and $\mathcal{R}_{2}$ gives bounds
for $\lambda_{\mathcal{R}}$, and hence for $\lambda(\tilde{S}^{-1}S)$ as above.
\end{proof}

We highlight that the case where $K$ is positive semidefinite rather than
positive definite is a more difficult one theoretically, as the expression
\eqref{eqn:equiv_schur} for $S\widetilde{S}^{-1}$ cannot be derived (it requires
the existence of $K^{-1}$). However it is clear that the spectral properties of
$K$ are almost identical for Dirichlet (positive definite) and Neumann (positive
semidefinite) problems, apart from the single zero eigenvalue in the
semidefinite setting, and we generally find the convergence rates of iterative methods
are very similar as a result. Indeed in practice we observe that
the eigenvalues of $\tilde{S}^{-1}S$ in our numerical experiments reflect the
predicted bounds very well.

We also note that we observe the eigenvalues of $M^{-1}M_E$ to be mesh
independent, and that the bounds for $\lambda(\tilde{S}^{-1}S)$ can
therefore be driven tighter by decreasing $\Delta{}t$. Therefore, as the
dimension of the matrix system is increased by taking a finer discretization in
space or time, we predict that our preconditioner should not worsen in
performance. Furthermore, if $\Delta t$ is chosen to scale like $\varepsilon^2$,
then the eigenvalue bounds asymptote to a constant as the interfacial thickness
parameter $\varepsilon \to 0$.  Hence, we also expect the preconditioner to be
robust to changes in this parameter.  Note that the scaling $\Delta t \sim
\varepsilon^2$ is typically necessary for stability in time discretization
schemes \cite{benesova2014}, and thus this criterion does not represent an
additional restriction on the timestep size.

In the next section we demonstrate how our proposed solver performs for a
practical test problem, and examine whether the predicted robustness is
achieved.

\section{Numerical results} \label{sec:experiments}
The solver was implemented in 338 lines of Python using version 1.6 of the FEniCS finite
element library \cite{logg2011}\footnote{The code is available at \url{http://bitbucket.org/pefarrell/ok-solver}.}.
The experiment conducted used $\Omega = (0, 1)^3$, $m = 0.4$, $\varepsilon =
0.02$, $\sigma = 100$, and an initial
condition of
\begin{equation}
u_0(x, y, z) = m + p(x, y, z),
\end{equation}
where the perturbation $p$ must be chosen to have integral zero and satisfy $\nabla p \cdot n = 0$ on
the boundary.
In this experiment we chose
\begin{equation}
p(x, y, z) = \frac{\cos{(2 \pi x)} \cos{(2 \pi y)} \cos{(2 \pi z)}}{50}.
\end{equation}
The initial condition for $w$ was computed from $u_0$ via \eqref{eqn:okb}.
In this parameter regime the solution is expected to consist of spherical
regions of negative material ($u \approx -1$) embedded within a background of
positive material ($u \approx 1$), Figure \ref{fig:simulation}.

\begin{figure}
\centering
\includegraphics[width=10cm]{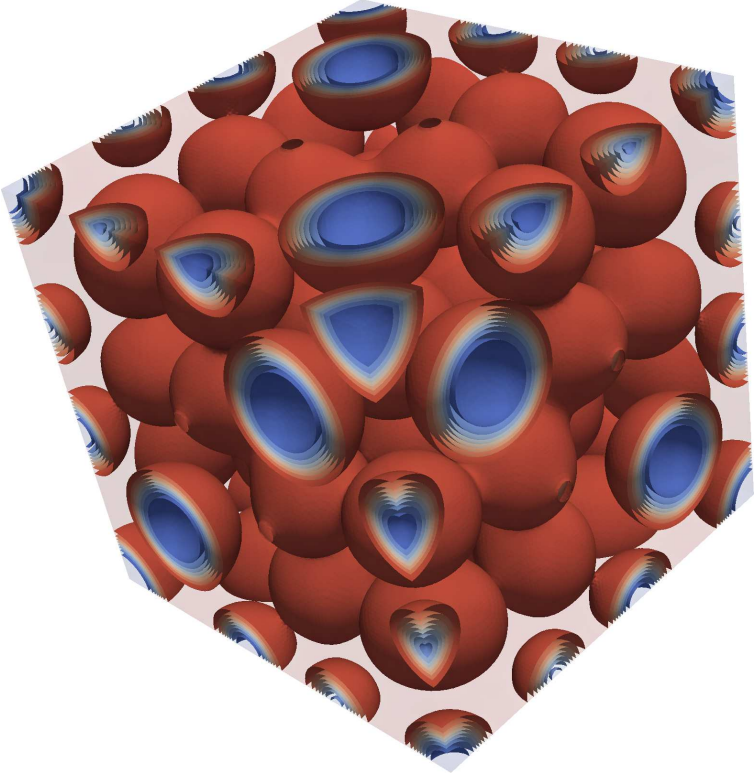}
\caption{The final concentration $u$ of the simulation described in section \ref{sec:experiments}. The
solution consists of regions of negative material ($u \approx -1$, in blue) embedded within
a background of positive material ($u \approx 1$, in red). The figure shows the isosurfaces corresponding
to constant values of $u$.}
\label{fig:simulation}
\end{figure}

The discretization used standard piecewise linear finite elements for $u$ and
$w$, a timestep $\Delta t = \varepsilon^2$, a final time $T = 300 \Delta t$,
and an implicitness parameter
$\theta = 0.5$. The mesh was generated to achieve approximately $2.5 \times 10^5$
degrees of freedom per core, to investigate how the number of
Krylov iterations required to solve \eqref{eqn:newton} varies as the mesh is
refined. The preconditioner \eqref{eqn:prec} approximated the action of $A^{-1}$
with ten Chebyshev semi-iterations with SOR preconditioning, and approximated the action of
$S^{-1}$ with two Richardson iterations of $\tilde{S}^{-1}$. Each action of
$\hat{S}^{-1}$ was in turn approximated with five V-cycles of the BoomerAMG algebraic
multigrid solver \cite{henson2002}. GMRES was used as the outer Krylov solver.

\begin{table}
\centering
\begin{tabular}{c|c|c}
\toprule
Degrees of freedom ($\times 10^6$) & Number of cores & Iterations \\
\midrule
0.265   & 1 & 8.2 \\
2.060   & 8 & 8.0 \\
16.24   & 64 & 8.0 \\
128.9   & 512 & 8.0 \\
1027    & 4096 & 8.0 \\
\bottomrule
\end{tabular}
\caption{Average number of GMRES iterations per Newton step with preconditioner
\eqref{eqn:prec} and Schur complement approximation \eqref{eqn:schurapprox} as
the problem is weakly scaled. The number of iterations required grows slowly as
the mesh is refined.}
\label{tab:results}
\end{table}

\begin{table}
\centering
\begin{tabular}{c|c|c|c}
\toprule
$\varepsilon$ & $\Delta x$ & $\Delta t$ & Iterations \\
\midrule
0.02  & 0.01 & 0.0004 & 8.0 \\
0.01  & 0.005&  0.0001 & 8.0 \\
0.005 &  0.0025 & 0.000025 & 8.0 \\
\bottomrule
\end{tabular}
\caption{Average number of GMRES iterations per Newton step with preconditioner
\eqref{eqn:prec} and Schur complement approximation \eqref{eqn:schurapprox} as
the interfacial thickness $\varepsilon$ is varied (and with it the spatial
discretization $\Delta x$ and timestep $\Delta t$). The performance of the
preconditioner does not change as $\varepsilon \to 0$.}
\label{tab:epsresults}
\end{table}

The essential properties of a good preconditioner are that the Krylov iteration
counts are low, they grow slowly (if at all) with mesh refinement, and they are
robust to variation in parameters. We therefore examined the average number of
Krylov iterations required per Newton step in two numerical experiments: in the
the first, all parameters were fixed, and only the mesh was refined; in the
second, the interfacial thickness $\varepsilon$ was reduced, and with it the
spatial discretisztion (to resolve the interface) and the temporal
discretization (to retain stability).  The experiments were conducted on
Hexagon, a Cray XE6m-200 hosted at the University of Bergen, and ARCHER, a Cray
XC30 hosted at the University of Edinburgh.

The results of the first experiment are shown in Table \ref{tab:results}. The
average number of Krylov
iterations per Newton step remains very close to 8, even as the number of
degrees of freedom is increased by four orders of magnitude. The results of the
second experiment are shown in Table \ref{tab:epsresults}. The average number of Krylov
iterations required per Newton step does not vary as $\varepsilon \to 0$.

\section{Conclusions}

We have presented a new preconditioner for the Ohta--Kawasaki equations that
model diblock copolymer melts. An approximation to the Schur complement was
derived using a matching strategy. The preconditioner proposed yields mesh
independent convergence, is robust to changes in interfacial thickness if a
timestep criterion is satisfied, and requires no reassembly or reconstruction
between Newton steps. This enables the solution of very fine discretizations
with billions of degrees of freedom.

\bibliographystyle{siam}
\bibliography{literature}
\end{document}